\documentclass[twoside]{article}  
\usepackage[utf8]{inputenc}

\usepackage{graphicx,color}

\usepackage[format=hang,
            font  ={footnotesize,sf},
            labelfont = {bf},
            margin =1cm,
            aboveskip = 5pt,
            position = bottom]{caption}

\usepackage{a4,times,latexsym,amssymb}
\usepackage[T1]{fontenc}
\usepackage{mathrsfs}
\usepackage{theorem}
\usepackage{amsmath}
\numberwithin{equation}{section}

\usepackage{latexsym}
\usepackage{xspace}
\usepackage{dsfont}
\usepackage{enumerate}
\usepackage[round]{natbib}

\newcommand{\nat}{\ensuremath{\mathbb N}}
\newcommand{\real}{\ensuremath{\mathbb R}}

\newcommand{\integer}{\ensuremath{\mathbb Z}}

\usepackage{theorem} 

\newtheorem{theorem}{Theorem}[section]

\newtheorem{propos}[theorem]{Proposition}

\newtheorem{definition}[theorem]{Definition}

\newtheorem{cor}[theorem]{Corollary}

\newtheorem{lemma}[theorem]{Lemma}

\newtheorem{Example}[theorem]{Example}

{\theorembodyfont{\rmfamily}\newtheorem{Example-rm}[theorem]{Example}}
{\theorembodyfont{\rmfamily}\newtheorem{Examples-rm}[theorem]{Examples}}

\newtheorem{Remark}[theorem]{Remark}

{\theorembodyfont{\rmfamily}\newtheorem{Remark-rm}[theorem]{Remark}}

\newcommand{\BOX}{\mbox{{\ensuremath{\Box}}\hspace{-0.5mm}}}

\newenvironment{proof}{{\par\noindent\bf Proof:}}{\mbox{}\hfill$\BOX$\\}

{\noindent\textbf{Proof of claim \protect\ref{#1}:}} {\mbox{}\hspace*{\fill}\\}

{\par\bigskip \noindent\textbf{Proof of Theorem \protect\ref{#1}:}} 
{\mbox{}\hspace*{\fill}$\BOX$\\}
{{\par\bigskip \noindent\bf{Proof of Theorems \protect\ref{#1}--\protect\ref{#2}:}}}
{\mbox{}\hfill\hspace*{\fill}$\BOX$\\}

\newcommand\refeq[1]{{{\rm (\ref{#1})}}}  

\newcommand{\bee}{\begin{equation}} 
\newcommand{\ene}{\end{equation}}

\newcommand{\beE}{\begin{Equation}} 
\newcommand{\enE}{\end{Equation}}

\newcommand{\bdi}{\begin{displaymath}}  
\newcommand{\edi}{\end{displaymath}}

\newcommand{\bDI}{\begin{Displaymath}}  
\newcommand{\eDI}{\end{Displaymath}}

\newcommand{\bqa}{\begin{eqnarray}} 
\newcommand{\eqa}{\end{eqnarray}}

\newcommand{\bqA}{\begin{Eqnarray}} 
\newcommand{\eqA}{\end{Eqnarray}}

\newcommand{\bea}{\begin{eqnarray*}}  
\newcommand{\ena}{\end{eqnarray*}}

\newcommand{\beA}{\begin{Eqnarray*}}  
\newcommand{\enA}{\end{Eqnarray*}}

\title{On optimal stationary couplings between stationary processes}
\author{Ludger Rüschendorf\thanks{University Freiburg}, Tomonari Sei\thanks{Keio University}}

\begin{document}

\maketitle\thispagestyle{empty}

\begin{abstract}
 By a classical result of \cite{Gray-Neuhoff-Shields-1975} the $\bar\varrho$ distance between stationary processes is identified with an optimal stationary coupling problem of the corresponding stationary measures on the infinite product spaces. This is a modification of the optimal coupling problem from Monge--Kantorovich theory. In this paper we derive some general classes of examples of optimal stationary couplings which allow to calculate the $\bar\varrho$ distance in these cases in explicit form. We also extend the $\bar\varrho$ distance to random fields and to general nonmetric distance functions and give a construction method for optimal stationary $\bar c$-couplings. Our assumptions need in this case a geometric positive curvature condition.
\end{abstract}

\renewcommand{\thefootnote}{}
\footnotetext{\hspace*{-.51cm}%
}

\section{Introduction}\label{sec-Intro}

\cite{Gray-Neuhoff-Shields-1975} introduced the $\bar\varrho$ distance between two stationary probability measures $\mu$, $\nu$ on $E^\integer$, where $(E,\varrho)$ is a separable, complete metric space (Polish space). The $\bar\varrho$ distance extends Ornstein's $\bar d$ distance (\cite{Ornstein-1973}) and is applied to the information theoretic problem of source coding with a fidelity criterion, when the source statistics are incompletely known. $\bar\varrho$ is defined via the following steps. Let $\varrho_n:E^n\times E^n\to\real$ denote the average distance per component on $E^n$ 
\begin{equation}\label{eq:1.1}
 \varrho_n(x,y) := \frac1n \sum_{i=0}^{n-1} \varrho(x_i,y_i),\quad x=(x_0,\dots,x_{n-1}), y=(y_0,\dots,y_{n-1}).
\end{equation}

Let $\bar\varrho_n$ denote the corresponding minimal $\ell_1$-metric also called Wasserstein distance or Kantorovich distance of the restrictions of $\mu$, $\nu$ on $E^n$, i.e.
\begin{equation}\label{eq:1.2}
 \bar\varrho_n(\mu,\nu) = \inf\bigg\{\int \varrho_n(x,y)d\beta(x,y)\mid \beta\in M(\mu^n,\nu^n)\bigg\},
\end{equation}
where $\mu^n$, $\nu^n$ are the restrictions of $\mu$, $\nu$ on $E^n$, i.e. on the coordinates $(x_0,\dots,x_{n-1})$ and $M(\mu^n,\nu^n)$ is the Fr\'echet class of all measures on $E^n\times E^n$ with marginals $\mu^n$, $\nu^n$. Then the $\bar\varrho$ distance between $\mu$, $\nu$ is defined as
\begin{equation}\label{eq:1.3}
 \bar\varrho(\mu,\nu) = \sup_{n\in\nat} \bar\varrho_n(\mu,\nu).
\end{equation}
It is known that $\bar\varrho(\mu,\nu)=\lim_{n\to\infty}\bar\varrho_n(\mu,\nu)$ by Fekete's lemma on superadditive sequences.

$\bar\varrho$ has a natural interpretation as average distance per coordinate between two stationary sources in an optimal coupling. In the original Ornstein version $\varrho$ was taken as discrete metric on a finite alphabet. This interpretation is further justified by the basic representation result (cp. \citet[Theorem 1]{Gray-Neuhoff-Shields-1975})
\begin{eqnarray}%
\label{eq:1.4}
\bar\varrho(\mu,\nu) %
&=& \bar\varrho_s (\mu,\nu) := \inf_{\Gamma\in M_{\rm s}(\mu,\nu)} \int \varrho(x_0,y_0)d\Gamma(x,y) 
\\
\label{eq:1.5}
&=& \inf \{ E\varrho(X_0,Y_0)\mid (X,Y) \sim \Gamma\in M_{\rm s}(\mu,\nu)\}.
\end{eqnarray}
Here $M_{\rm s}(\mu,\nu)$ is the set of all jointly stationary (i.e. jointly shift invariant) measures on $E^\integer\times E^\integer$ with marginals $\mu$, $\nu$ and $(X,Y)\sim \Gamma$ means that $\Gamma$ is the distribution of $(X,Y)$. Thus $\bar\varrho(\mu,\nu)$ can be seen as a Monge--Kantorovich problem on $E^\integer$ with however a modified Fr\'echet class $M_{\rm s}(\mu,\nu)\subset M(\mu,\nu)$. \eqref{eq:1.5} states this as an optimal coupling problem between jointly stationary processes $X$, $Y$ with marginals $\mu$, $\nu$.  A pair of jointly stationary processes $(X,Y)$ with distribution $\Gamma\in M_{\rm s}(\mu,\nu)$ is called \emph{optimal stationary coupling of $\mu,\nu$} if it solves problem \eqref{eq:1.5}, i.e. it minimizes the stationary coupling distance $\bar\varrho_s$.

By definition it is obvious (see \cite{Gray-Neuhoff-Shields-1975}) that
\begin{eqnarray}\label{eq:1.6}
 \bar\varrho_1(\mu,\nu)\le \bar\varrho(\mu,\nu) %
\le \int\varrho(x_0,y_0) d\mu^0(x_0) d\nu^0(y_0),
\end{eqnarray}
the left hand side being the usual minimal $\ell_1$-distance (Kantorovich distance) between the single components $\mu^0$, $\nu^0$.

As remarked in \citet[Example 2]{Gray-Neuhoff-Shields-1975} the main representation result in \eqref{eq:1.4}, \eqref{eq:1.5} does not use the metric structure of $\varrho$ and $\varrho$ can be replaced by a general cost function $c$ on $E\times E$ implying then the generalized optimal stationary coupling problem
\begin{equation}\label{eq:1.7}
 \bar c_s (\mu,\nu) = \inf\{E c (X_0,Y_0)\mid (X,Y)\sim\Gamma\in M_{\rm s}(\mu,\nu)\}.
\end{equation}

Only in few cases information on this optimal coupling problem for $\bar\varrho$ resp. $\bar c$ is given in the literature. \cite{Gray-Neuhoff-Shields-1975} determine $\bar\varrho$ for two i.i.d. binary sequences with success probabilities $p_1$, $p_2$. They also derive for quadratic cost $c(x_0,y_0)=(x_0-y_0)^2$ upper and lower bounds for two stationary Gaussian time series in terms of their spectral densities. We do not know of further explicit examples in the literature for the $\bar\varrho$ distance. The aim of our paper is to derive optimal couplings and solutions for the $\bar\varrho$ metric resp. the generalized $\bar c$ distance.

The $\bar\varrho$ resp. $\bar c$ distance is particularly adapted to stationary processes. One should note that from the general Monge--Kantorovich theory characterizations of optimal couplings for some classes of distances $c$ are available and have been determined for time series and stochastic processes in some cases. For processes with values in a Hilbert space (like the weighted $\ell_2$ or the weighted $L^2$ space) and for general cost functions $c$, general criteria for optimal couplings have been given in \cite{Rueschendorf-Rachev-1990} and \cite{Rueschendorf-1991}. For some examples and extensions to Banach spaces see also \cite{CuestaAlbertos-Rueschendorf-TueroDiaz-1993}
and \cite{Rueschendorf-1995}.  Some of these criteria have been further extended to measures $\mu$, $\nu$ in the Wiener space $(W,H,\mu)$ w.r.t. the squared distance $c(x,y)=|x-y|_H^2$ by Feyel and \"Ust\"unel (2002, 2004)\nocite{Feyel-Uestuenel-2002} \nocite{Feyel-Uestuenel-2004} and \cite{Uestuenel-2007}. All these results are also applicable to stationary measures and characterize optimal couplings between them. But they do not respect the special stationary structure as described in the representation result in \eqref{eq:1.5}, \eqref{eq:1.7}. In the following sections we want to determine optimal stationary couplings between stationary processes.

In Section \ref{sec:2} we consider the optimal stationary coupling of stationary processes on $\real$ and on $\real^m$ with respect to squared distance. In Section \ref{sec:3} we give an extension to the case of random fields. Finally we consider in Section \ref{sec:4} an extension to general cost functions. We interpret an optimal coupling condition by a geometric curvature condition.

\section{Optimal couplings of stationary processes w.r.t. squared distance}
\label{sec:2}

In this section we consider the case where $E=\real$ (resp. $\real^m$), $\Omega=E^\integer$ and with squared distance $c(x_0,y_0)=(x_0-y_0)^2$ (resp. $\|x_0-y_0\|^2$ on $\real^m$). Let $L:\Omega\to\Omega$ denote the left shift, $(Lx)_t=x_{t-1}$. Then a pair of processes $(X,Y)$  with values in $\Omega\times\Omega$ is \emph{jointly stationary} when $(X,Y)\stackrel{d}{=} (LX,LY)$ ($\stackrel{d}{=}$ denotes equality in distribution). A Borel measurable map $S:\Omega\to\Omega$ is called \emph{equivariant} if
\begin{equation}\label{eq:2.1}
 L\circ S=S\circ L.
\end{equation}
This notion is borrowed from the corresponding notion in statistics, where it is used in connection with statistical group models. The following lemma concerns some elementary properties.

\begin{lemma}\label{lem:2.1}
\begin{enumerate}[a)]
 \item A map $S:\Omega\to\Omega$ is equivariant if and only if $S_t(x)=S_0(L^{-t}x)$ for any $t,x$.
 \item If $X$ is a stationary process and $S$ is equivariant then $(X,S(X))$ is jointly stationary.
\end{enumerate}

\end{lemma}

\begin{proof}
 \begin{enumerate}[a)]
  \item 
If $L\circ S=S\circ L$ then by induction $S=L^t\circ S\circ L^{-t}$ for all $t\in \integer$, and thus $S_t(x)=S_0(L^{-t}x)$. Conversely, if $S_t(x)=S_0(L^{-t}x)$, then $S_{t-1}(x)=S_0(L^{-t+1}x)=S_t(Lx)$. This implies $L(S(x))=S(Lx)$. 
\item 
Since $LX$ has the same law as $X$, it follows that $(LX,L(S(X)))=(LX,S(LX))=(I,S)(LX)\stackrel{d}{=} (I,S)(X)=(X,S(X))$, I denoting the identity.
\vspace*{-3.75ex}
 \end{enumerate}
\end{proof}

For $X\stackrel{d}{=} \mu$ and $S:\Omega\to\Omega$ the pair $(X,S(X))$ is called \emph{optimal stationary coupling} if it is an optimal stationary coupling w.r.t. $\mu$ and $\nu:=\mu^S=S_\#\mu$, i.e., when $\nu$ is the corresponding image (push-forward) measure.

We first consider the case $E=\real$ and $\Omega=\real^\integer$.
To construct a class of optimal stationary couplings we define for a convex function $f:\real^n\to\real$ an equivariant map $S:\Omega\to\Omega$. For $x\in\Omega$ let
\begin{equation}\label{eq:2.2}
 \partial f(x)=\{y\in \real^n\mid f(z)-f(x)\ge y\cdot (z-x), \quad \forall z\in\real \}
\end{equation}
denote the subgradient of $f$ at $x$, where $a\cdot b$ denotes the standard inner product of vectors $a$ and $b$. By convexity $\partial f(x)\not= \phi$. Let $F(x)= (F_k(x))_{0\le k\le n-1}$ be measurable and $F(x)\in\partial f(x)$, $x\in\real^n$. The equivariant map $S$ is defined via Lemma~\ref{lem:2.1} by 
\begin{equation}\label{eq:2.3}
 S_0(x) = \sum_{k=0}^{n-1} F_k(x_{-k},\dots, x_{-k+n-1}), \qquad  S_t(x) = S_0(L^{-t}x), \quad x\in\Omega.
\end{equation}

For terminological reasons we write any map of the form \eqref{eq:2.3} as
\begin{equation}\label{eq:2.4}
 S_0(x) = \sum_{k=0}^{n-1} \partial_k f(x_{-k},\dots, x_{-k+n-1}), \qquad  S_t(x) = S_0(L^{-t}x), \quad x\in\Omega.
\end{equation}
In particular for differentiable convex $f$ the subgradient set coincides with the derivative of $f$, $\partial f(x)=\{\nabla f(x)\}$ and $\partial_t f(x)=\frac{\partial}{\partial x_t} f(x)$.

\medskip

\begin{Remark}\label{rem:2.1}
 \begin{enumerate}[a)]
\item %
In information theory a map of the form $S_t(x)=F(x_{t-n+1},\dots, x_{t+n-1})$ is called a sliding block code (see \cite{Gray-Neuhoff-Shields-1975}). Thus our class of maps $S$ defined in \eqref{eq:2.4} are particular sliding block codes.
\item %
\cite{Sei-2006, Sei-2010a, Sei-2010b} introduced so-called \emph{structural gradient models (SGM)} for stationary time series, which are defined as $\{(S_\vartheta)^\# Q\mid\vartheta\in\Theta\}$, where $Q$ is the infinite product of the uniform distribution on $[0,1]$, on $[0,1]^\integer$, $\{S_\vartheta\mid\vartheta\in \Theta\}$ is a parametric family of transformations of the form given in \eqref{eq:2.4} and $S_\vartheta^\# Q$ denotes the pullback measure of $Q$ by $S_\vartheta$. 
It turns out that these models have nice statistical properties, e.g. they allow for simple likelihoods and allow the construction of flexible dependencies. 
The restriction to functions of the form \eqref{eq:2.4} is well founded by an extended Poincar\'e lemma (see \citet[Lemma 3]{Sei-2010b}) saying in the case of differentiable $f$ that these functions are the only ones with (the usual) symmetry and with an additional stationarity property $S_{t-1}(x)=S_t(Lx)$ for $x\in\real^\integer$, which is related to our notion of equivariant mappings.
\item %
Even if a map $S$ has a representation of the form \eqref{eq:2.4}, the inverse map $S^{-1}$ does not have the same form in general. We give an example. Let $X=(X_t)_{t\in\integer }$ be a real-valued stationary process  with a spectral representation  $X_t = \int_0^1 \mathrm{e}^{2\pi\mathrm{i}\lambda t}M(\mathrm{d}\lambda)$,
 where $M(\mathrm{d}\lambda)$ is an $L^2$-random measure. Define a process $Y=(Y_t)$ by
\[
  Y_t = S_t(X) := X_t + \epsilon(X_{t-1}+X_{t+1}),
 \quad \epsilon\neq 0.
 \]
This is of the form \eqref{eq:2.4} with a function $f(x_0,x_1) =x_0^2/4 +\epsilon x_0x_1+x_1^2/4$ which is convex if $|\epsilon|<1/2$. Under this condition, the map $X\mapsto Y$ is shown to be invertible as follows. The spectral representation of $Y$ is $N(\mathrm{d}\lambda) := (1+\epsilon(\mathrm{e}^{2\pi\mathrm{i}\lambda} + \mathrm{e}^{-2\pi\mathrm{i}\lambda})) M(\mathrm{d}\lambda)$.
 Then we have the following inverse representation
\[
  X_t = \int_0^1\frac{\mathrm{e}^{2\pi\mathrm{i}\lambda t}}
 {1+\epsilon(\mathrm{e}^{2\pi\mathrm{i}\lambda}+\mathrm{e}^{-2\pi\mathrm{i}\lambda})}
 N(\mathrm{d}\lambda)
 = \sum_{s\in\integer }b_sY_{t-s},
 \]
where $(b_s)_{s\in\integer }$ is defined by
 $\{1+\epsilon(\mathrm{e}^{2\pi\mathrm{i}\lambda} +\mathrm{e}^{-2\pi\mathrm{i}\lambda})\}^{-1}
 = \sum_{s\in\integer } b_s \mathrm{e}^{-2\pi\mathrm{i}\lambda s}$.
 By standard complex analysis, the coefficients $(b_s)$ are
 explicitly obtained:
\[
  b_s 
   = 
   \frac{z_+^{|s|}}{\epsilon(z_+-z_-)},
  \quad 
  z_{\pm} := \frac{-1\pm \sqrt{1-4\epsilon^2}}{2\epsilon}.
 \]
 Note that $|z_+|<1$ and $|z_-|>1$ since $|2\epsilon|<1$.
 Hence $b_s\neq 0$ for all $s\in\integer $ and the inverse map $S^{-1}(Y) = \sum_{s}b_sY_s$ does not have a representation as in  \eqref{eq:2.4}.
 \end{enumerate}
\end{Remark}

The following theorem implies that the class of equivariant maps defined in \eqref{eq:2.4} gives a class of examples of optimal stationary couplings between stationary processes.

\begin{theorem}[Optimal stationary couplings of stationary processes on \boldmath $\real$] \label{theo:2.2}
 Let $f$ be a convex function on $\real^n$, let $S$ be the equivariant map defined in \eqref{eq:2.4} and let $X$ be a stationary process with law $\mu$.
 Assume that $X_0$ and $\partial_kf(X^n)$ ($k=0,\ldots,n-1$) are in $L^2(P)$.
 Then $(X,S(X))$ is an optimal stationary coupling w.r.t. squared distance between $\mu$ and $\mu^S$, i.e.
\[
 E[(X_0-S_0(X))^2]=\min_{(X,Y)\sim\Gamma\in M_{\rm s}(\mu,\mu^S)} E[(X_0-Y_0)^2] = \bar\varrho_{\rm s} (\mu,\mu^S),
\]

\end{theorem}

\begin{proof}
 Fix any $\Gamma\in M_{\rm s}(\mu,\mu^S)$.
 By the gluing lemma (see Appendix \ref{sec:A}),
 we can construct a jointly stationary process $(X,Y,\tilde{X})$ on a common probability space
 such that $X\sim \mu$, $Y=S(X)$ and $(\tilde{X},Y)\sim \Gamma$.
 From the definition of $Y_0=S_0(X)$, we have $Y_0\in L^2(P)$.
 Then by the assumption of identical marginals
%
%
%
 \begin{eqnarray*}
   A &:=& \frac{1}{2}\mathrm{E}[(X_0-Y_0)^2-(\tilde{X}_0-Y_0)^2] 
  \\
  &\kern.5ex=& \mathrm{E}[-X_0Y_0+\tilde{X}_0Y_0] 
  \\
  &\kern.5ex=& \mathrm{E}[(\tilde{X}_0-X_0)S_0(X)]
  \\
  &\kern.5ex=& \mathrm{E}\left[
  (\tilde{X}_0-X_0)\sum_{k=0}^{n-1}  (\partial_kf)(X_{-k},\ldots,X_{-k+n-1})
  \right].
\end{eqnarray*}

Using the stationarity assumption on $X$ we get with $X^n=(X_0,\dots X_{n-1})$, $\widetilde X^n=(X_0,\dots,\widetilde X_{n-1})$ that
\begin{eqnarray*}
   A  
  &=& \mathrm{E}\left[
  \sum_{k=0}^{n-1}(\tilde{X}_k-X_k) (\partial_kf)(X_0,\ldots,X_{n-1})
  \right]\\
&\leq& \mathrm{E}[f(\tilde{X}^n)-f(X^n)] 
  \\
  &=& 0,
 \end{eqnarray*}
the inequality is a consequence of convexity of $f$. This implies optimality of $(X,Y)$. We note that the last equality uses integrability of $f(X^n)$, which comes from convexity of $f$ and the $L^2$-assumptions. This completes the proof.
\end{proof}

Theorem \ref{theo:2.2} allows to determine explicit optimal stationary couplings for a large class of examples. Note that -- at least in principle -- the $\overline\varrho$ distance can be calculated in explicit form for this class of examples.

The construction of Theorem \ref{theo:2.2} can be extended to multivariate stationary sequences in the following way. Let $(X_t)_{t\in\integer}$ be a stationary process, $X_t\in\real^m$ and let $f:(\real^m)^n\to\real$ be a convex function on $(\real^m)^n$. Define an equivariant map $S:(\real^m)^\integer  \to (\real^m)^\integer$ by
\begin{equation}\label{eq:2.5}
 \begin{split}
  S_0(x) &=\sum_{k=0}^{n-1} \partial_k f(x_{-k},\dots,x_{-k+n-1})\\
S_t(x) &= S_0(L^{-t}x), \quad x\in\Omega=(\real^m)^\integer
 \end{split}
\end{equation}
where $L^{-t}$ operates on each component of $x$ and $\partial_\ell f$ is (a representative of) the subgradient of $f$ w.r.t. the $\ell$-th component. Thus for differentiable $f$ we obtain 
\begin{equation}\label{eq:2.6}
 S_0(x) = \sum_{k=0}^{n-1} \nabla_k f(x_{-k},\dots,x_{-k+n-1})
\end{equation}
where $\nabla_\ell f$ is the gradient of $f$ w.r.t. the $\ell$-th component.

The classical result for optimal couplings w.r.t. the squared norm distance on $\real^m$ due to \cite{Rueschendorf-Rachev-1990} and \cite{Brenier-1991} characterizes optimal couplings $(Y,Z)$ of distributions $P$, $Q$ on $\real^m$ by the condition that
\begin{equation}\label{eq:2.7}
 Z\in\partial h(Y) \enskip a.s.
\end{equation}
for some convex function $h$. The construction in \eqref{eq:2.5} adapts this result to optimal stationary couplings of stationary processes on $\real^m$.

\begin{theorem}[Optimal stationary couplings of stationary processes on \boldmath $\real^m$]\label{theo:2.3}
 ~ Let $f$ be a convex function on $(R^m)^n$ and let $S$ be the equivariant map on $\Omega=(\real^m)^\integer$ defined in \eqref{eq:2.5}. Let $X$ be a stationary process on $\real^m$ with distribution $\mu$ and assume that $X_0$ and $\partial_k f(X^n)$, $0\le k\le n-1$, are square integrable. Then $(X,S(X))$ is an optimal stationary coupling between $\mu$ and $\mu^S=S_\#\mu$ w.r.t. squared distance, i.e.
\begin{equation}\label{eq:2.8}
 E\|X_0-S_0(X)\|_2^2 %
= \inf \{E\|Y_0-Z_0\|_2^2 \mid (Y,Z)\sim \Gamma\in M_{\rm s}(\mu,\mu^S)\}%
= \bar\varrho_{\rm s} (\mu,\mu^S).
\end{equation}
\end{theorem}

\begin{proof}
 The proof is similar to that of Theorem \ref{theo:2.2}. For a jointly stationary process $(X,Y,\tilde X)$ with $X\sim \mu$, $Y\stackrel{d}{=} S(X)$ and $\tilde X\stackrel{d}{=} X\sim \mu$ we have using stationarity and convexity as in Theorem~\ref{theo:2.2}.
\begin{eqnarray*}
 \frac12 E(\|X_0-Y_0\|_2^2 - \|\tilde X_0-Y_0\|_2^2) 
&=& E[-X_0\cdot Y_0+\tilde X_0\cdot Y_0] \\
&=& E(\tilde X_0-X_0)\cdot \sum_{k=0}^{n-1} \partial_k f(X_{-k},\dots,X_{-k+n-1})
\\
&=& E\sum_{k=0}^{n-1} (\tilde X_k-X_k) \cdot \partial_k f(X_0,\dots,X_{n-1}) 
\\
&\le& E(f(\tilde X_0,\dots,\tilde X_{n-1})-f(X_0,\dots,X_{n-1})) = 0.
\end{eqnarray*}
The third equality follows from the stationarity assumption and the inequality follows from convexity of $f$. Thus \eqref{eq:2.8} follows.
\end{proof}

\begin{Remark}\label{rem:2.2}
Considering the case where $\mu$ is a stationary probability measure on $\real^\integer$ corresponding to the real stationary process $X$ on $\real$ we can introduce the multivariate stationary process $Y$ by $Y_k=(X_k,X_{k+1},\dots,X_{k+m-1})$ on $\real^m$. As consequence of Theorem~\ref{theo:2.3} we obtain explicit optimal coupling results for the strengthened stationary distances relative to \eqref{eq:1.3}, \eqref{eq:1.4}, \eqref{eq:1.5} by comparing finite dimensional distributions
\begin{equation}\label{eq:2.9}
\begin{split}
\varrho^m(\mu,\nu)  = \inf\big\{E\|Y_0-Z_0\|^2  \mid Y_0  
& \stackrel{d}{=} \mu^m, Z_0\stackrel{d}{=} \nu^m, \\
& (Y,Z) \text{ jointly stationary}, Y\stackrel{d}{=} \mu, Z\stackrel{d}{=}\nu \big\}
\end{split}
\end{equation}
\end{Remark}

Thus we can compare and optimally couple not only the one-dimensional marginals in a stationary way but can also compare the multivariate marginals in a stationary way.

\section{Optimal stationary couplings of random fields}
\label{sec:3}

In the first part of this section we introduce the $\bar\varrho$ distance defined on a product space in the case of countable groups and establish an extension of the \cite{Gray-Neuhoff-Shields-1975} representation result to random fields. In a second step we extend this result to amenable groups on a Polish function space. This motivates the consideration of the optimal stationary coupling result as in Section \ref{sec:2}.

We consider stationary real random fields on an abstract group $G$. Section \ref{sec:2} was concerned with the case of stationary discrete time processes, where $G=\integer$. Interesting extensions concern the case of stationary random fields on lattices $G=\integer^d$ or the case of stationary continuous time stochastic processes with $G=\real$ or $G=\real^d$.

Let $e$ be the unit element of $G$.
We consider the product space $\Omega=E^ G $ of a Polish space $(E,\varrho)$ (e.g. $E=\real$) equipped with the product topology. 
Note that $\Omega$ is not Polish in general, but its marginal sets $E^F$ on a finite or countable subset $F\subset G$ are Polish.
The (left) group action of $ G $ on $\Omega$ is defined by $(gx)_h=x_{g^{-1}h}$.
In particular, $(gx)_g=x_e$.
The function $x\mapsto gx$ is continuous.
A Borel probability measure $\mu$ on $\Omega$ is called stationary if $\mu^g=\mu$ for every $g\in G$.

Let $P$ and $Q$ be stationary Borel probability measures on $\Omega=E^G$. 
For any finite subset $F$ of $G$ and sequences $x_F=(x_g)_{g\in F}$ and $y_F=(y_g)_{g\in F}$,
define $\varrho_F(x_F,y_F)=|F|^{-1}\sum_{g\in F}\varrho(x_g,y_g)$.
Define $\bar{\varrho}_F(P,Q)$ by 
\begin{equation}\label{eq:3.0a}
 \bar{\varrho}_F(P,Q) = \inf_{(X_F,Y_F)\sim\Gamma_F\in M(P_F,Q_F)}\mathrm{E}[\varrho_F(X_F,Y_F)],
\end{equation}
where $P_F$ and $Q_F$ are marginal distributions of $P$ and $Q$, respectively. The natural extension of the $\bar\varrho$ distance is defined by
\begin{equation}\label{eq:3.0b}
 \bar{\varrho}(P,Q)
  = \sup_{F\subset G}\bar{\varrho}_F(P,Q),
\end{equation}
where the supremum
is taken over all finite subsets $F$ of $G$.
We also define the stationary coupling distance $\bar\varrho_{\mathrm s}$
\begin{equation}\label{eq:3.0c}
 \bar{\varrho}_{\mathrm s} (P,Q)
 = \inf_{(X,Y)\sim \Gamma\in M_{\rm s}(P,Q)}\mathrm{E}[\varrho(X_e,Y_e)],
\end{equation}
where $M_{\rm s}(P,Q)$ is the set of jointly stationary measures
with marginals $P$ and $Q$.

\cite{Gray-Neuhoff-Shields-1975} showed that $\bar\varrho=\bar\varrho_{\mathrm s}$ if $G=\integer$ (see \eqref{eq:1.5}).
We will prove this equality for general countable groups $G$ under a weak kind of amenability  assumption.
In this section, we denote $\Gamma[\varrho]=\mathrm{E}[\varrho(X_e,Y_e)]$
and $\Gamma[\varrho_F]=\mathrm{E}[\varrho_F(X_F,Y_F)]$ for $\Gamma\in M(P,Q)$.

\begin{lemma} \label{lem:b1}
 $\bar{\varrho}(P,Q)\leq \bar{\varrho}_{\mathrm s} (P,Q)$.
\end{lemma}

\begin{proof}
 Fix an arbitrary $\epsilon>0$.
 Take a jointly stationary measure $\Gamma\in M_{\rm s}(P,Q)$
 such that $\Gamma[\varrho]\leq\bar{\varrho}_{\mathrm s} (P,Q)+\epsilon$.
 Then $\bar{\varrho}_F(P,Q)\leq \Gamma[\varrho_F]=\Gamma[\varrho]\leq\bar{\varrho}_{\mathrm s}(P,Q)+\epsilon$.
 Since $F$ and $\epsilon$ are arbitrary, we obtain $\bar{\varrho}(P,Q)\leq \bar{\varrho}_{\mathrm s} (P,Q)$.
\end{proof}

We need a technical lemma.

\begin{lemma} \label{lem:b2}
  Let $G$ be countable and $F\subset G$ be finite.
  Then
  \begin{align*}
 \bar{\varrho}_F(P,Q) = \inf_{\Gamma\in M(P,Q)}\Gamma[\varrho_F].
  \end{align*}
\end{lemma}

\begin{proof}
  It is sufficient to prove existence of $\Gamma\in M(P,Q)$
  for any $\Gamma_F\in M(P_F,Q_F)$.
  This follows from the general extension property of probability measures with given marginals.
\end{proof}

To establish the equality $\bar{\varrho}=\bar{\varrho}_{\mathrm{s}}$,
we put an additional amenability assumption on $G$. The proof of the following representation theorem follows the lines of the proof of Theorem~1 of \cite{Gray-Neuhoff-Shields-1975}.

\begin{theorem} \label{theo:3.3}
 Let $G$ be a countable group.
 Assume that there exists a sequence $\{F_n\}_{n\geq 0}$ of finite subsets of $G$
 such that $\lim_{n\to\infty}|F_n\cap (h F_n)|/|F_n|=1$ for any $h\in G$.
 Then 
\[
 \bar{\varrho}(P,Q)=\bar{\varrho}_{\mathrm{s}} (P,Q).
\]
\end{theorem}

\begin{proof}
 Fix $\epsilon>0$.
 For each $n\geq 0$, choose a measure $\Gamma_n\in M(P,Q)$
 such that $\Gamma_n[\varrho_{F_n}]\leq \bar{\varrho}_{F_n}+\epsilon$ (see Lemma \ref{lem:b2}).
 Define measures $\bar{\Gamma}_n$ by
 \begin{align*}
  \bar{\Gamma}_n(A)
  &= \frac{1}{|F_n|}\sum_{g\in F_n}\Gamma_n(gA).
 \end{align*}
 Note that $\bar{\Gamma}_n[\varrho]=\Gamma_n[\varrho_{F_n}]$.
 The first marginal measure of $\bar{\Gamma}_n$ is
 \begin{align*}
  \bar{\Gamma}_n(A_1\times \Omega)
  &= \frac{1}{|F_n|}\sum_{g\in F_n}
  \Gamma_n(g(A_1\times \Omega))
  = \frac{1}{|F_n|}\sum_{g\in F_n}P(gA_1)
  = P(A_1),
 \end{align*}
 since $P$ is stationary.
 Similarly, the second marginal measure of $\bar{\Gamma}_n$ is $Q$.
 Hence $\bar{\Gamma}_n\in M(P,Q)$.
 Since $P$ and $Q$ are tight measures, the sequence $\{\bar{\Gamma}_n\}_{n\geq 0}$
 is tight and therefore has a subsequence converging weakly.
 We assume without loss of generality that $\{\bar{\Gamma}_n\}_{n\geq 0}$ itself  converges weakly to a measure $\bar{\Gamma}$.
 Then $\bar{\Gamma}\in M(P,Q)$. Furthermore, $\bar\Gamma$ is stationary, i.e. $\bar\Gamma\in M_{\mathrm{s}} (P,Q)$. Indeed,
 for any $h\in G$ and measurable $A\subset \Omega^2$, we have
 \begin{align*}
  \bar{\Gamma}_n(hA)
  &= \frac{1}{|F_n|}\sum_{g\in F_n}\Gamma_n(ghA)
  \\
  &= \frac{1}{|F_n|}\sum_{g\in F_n\cap (h F_n)}\Gamma_n(gA) + \mathrm{o}(1)
  \\
  &= \bar{\Gamma}_n(A) + \mathrm{o}(1),
 \end{align*}
 where we used $\lim_{n\to\infty}|F_n\cap (h F_n)|/|F_n|=1$.
 This implies stationarity of $\bar{\Gamma}$.
 Finally,
 \begin{align*}
  \bar{\varrho}_{\mathrm{s}}
  &\leq \bar{\Gamma}[\varrho]
  \leq \limsup_{n\to\infty}\bar{\Gamma}_n[\varrho]
  = \limsup_{n\to\infty}\Gamma_n[\varrho_{F_n}]
  \leq \limsup_{n\to\infty}\bar{\varrho}_{F_n} + \epsilon
  \leq \bar{\varrho}+\epsilon.
 \end{align*}
 Since $\epsilon$ is arbitrary, we have $\bar{\varrho}_{\mathrm{s}} \leq\bar{\varrho}$.
\end{proof}

\begin{Remark}\label{rem:3.4}
\begin{enumerate}
 \item 
 For the example $G=\mathbb{Z}^d$, we can take $F_n=\{-n,\ldots,n\}^d$.
 On the other hand, if $G$
 is the free group generated by two elements $f_1,f_2\neq e$,
 then there does not exist a sequence $\{F_n\}$ satisfying the amenability condition
 because the neighboring set $(f_1F_n\cup f_2F_n\cup f_1^{-1}F_n\cup f_2^{-1}F_n)\setminus F_n$
 has at least $2|F_n|+2$ elements.
\item 
The above given proof extends directly to the case of compact groups where $\bar\Gamma_n$ is defined via integration w.r.t. the normalized Haar measure. An extension of the representation result to general amenable groups on product spaces seems possible, but there are still some technical problems. 
Instead we will give an extension to amenable transformation groups acting on Polish function spaces. 
\end{enumerate}
\end{Remark}

Let $(G,\mathcal G)$ be a group of measurable transformations acting on a Polish space $(B,\varrho)$ of real functions on $E$ and let $P$, $Q$ be stationary probability measures on $B$, i.e. $P^g=P$, $Q^g=Q$, $\forall g\in G$. We assume that $G$ is an amenable group, i.e. there exists a sequence $\lambda_n$ of asymptotically left invariant probability measures on $G$ such that 
\begin{equation}\label{eq:rue-1}
 \lambda_n(gA)-\lambda_n(A)\to 0, \quad\forall A\in \mathcal G.
\end{equation}
The hypothesis of amenability is central for example in the theory of invariant tests. Many of the standard transformation groups are amenable. A typical exception is the free group of two generators. The Ornstein distance can be extended to this class of stationary random fields as follows. Define the average distance w.r.t. $\lambda_n$ by
\begin{equation}\label{eq:rue-2}
 \varrho_n(x,y):=\int \varrho(gx,gy)\lambda_n(dg).
\end{equation}

The induced minimal probability metric is given by 
\begin{equation}\label{eq:rue-3}
 \bar\varrho_n(P,Q) = \inf\{E\varrho_n(X,Y)\mid (X,Y)\sim\Gamma\in M(P,Q) \}.
\end{equation}
Finally, the natural extension of the $\bar\varrho$ metric of \cite{Gray-Neuhoff-Shields-1975} is defined as 
\begin{equation}\label{eq:rue-4}
 \bar\varrho(P,Q) = \sup_n \bar\varrho_n (P,Q).
\end{equation}

\begin{Remark}\label{rem:3.5}
 In the particular case when $G$ is countable and $\lambda_n=\frac{1}{|F_n|} \sum_{g\in F_n} \varepsilon_g$ for some increasing class of finite sets $F_n\subset G$ we can take the product space $B=E^G$ and we obtain $\varrho_n(x,y)=\frac{1}{|F_n|} \sum_{g\in F_n} \varrho(gX_g, gY_g)$ and 
\begin{equation}\label{eq:rue-5}
 \bar\varrho_n(P,Q) = \inf\{E\varrho_n(X_{F_n}, Y_{F_n}\mid (X_{F_n}, Y_{F_n})\sim\Gamma_{F_n} \in M(P_{F_n}, Q_{F_n})\}
\end{equation}
with $X_{F_n}=(gX)_{g\in F_n}=:\pi_{F_n}(X)$,  $Y_{F_n}=(gY)_{g\in F_n} = \pi_{F_n}(Y)$. Thus $\bar\varrho_n$ depends only on the \emph{finite dimensional projections} $P_{F_n}=P^{\pi_{F_n}}$, $Q_{F_n}=Q^{\pi_{F_n}}$ of $P$, $Q$ and we include the previous framework. Amenability of $G$ corresponds to the condition that $F_n$ is asymptotically left invariant in the sense that
\begin{equation}\label{eq:rue-6}
 |F_n\cap (hF_n)|/|F_n| \to 1,\quad\forall h\in G,
\end{equation}
i.e. to the condition in Theorem \ref{theo:3.3}.
\end{Remark}

The optimal stationary coupling problem is introduced similarly as in Section \ref{sec:2} by
\begin{equation}\label{eq:rue-7}
 \bar\varrho_s(P,Q)=\inf\{E[\varrho(eX,eY)]\mid (X,Y)\sim \Gamma\in M_s(P,Q)\}
\end{equation}
where $M_s(P,Q)=\{\Gamma\in M^1(B\times B)\mid \Gamma^{(g,g)}=\Gamma,\enskip \forall g\in G\}$ is the class of jointly stationary measures with marginals $P$, $Q$ and $e$ is the neutral element of $G$. We use the notation $\Gamma(\varrho)=E[\varrho(eX,eY)]$ and $\Gamma_n(\varrho)=E[\varrho_n(X,Y)]$ for $\Gamma\in M(P,Q)$.

We now can state an extension of the Gray--Neuhoff--Shields representation result for the $\bar\varrho$ distance of stationary random fields to amenable groups.

\begin{theorem}[General representation result for \boldmath$\bar\varrho$ distance]%
\label{theo:rue-1}
~  Let $G$ be an amenable\linebreak group acting on a Polish function space $B$ on $E$, let $P$, $Q$ be stationary integrable probability measures on $B$, i.e. for $X\stackrel{d}{=} P$, $E\varrho(X,y)<\infty$ for $y\in E$. Then the extended Ornstein distance $\bar\varrho$ defined in \eqref{eq:rue-4} coincides with the optimal stationary coupling distance $\bar\varrho_{\mathrm s}$, 
\[
\bar\varrho(P,Q)=\bar\varrho_{\mathrm s}(P,Q).                                                                                                                                                                                                                                                                                                                                      \]
In particular, $\bar\varrho$ does not depend on choice of $\lambda_n$.
\end{theorem}

\begin{proof}
 To prove that $\bar\varrho(P,Q)\le \bar\varrho_{\mathrm s}(P,Q)$ let for $\varepsilon>0$ given $\Gamma\in M_{\mathrm s}(P,Q)$ be such that $\Gamma(\varrho)\le \bar\varrho_{\mathrm s}(P,Q)+\varepsilon$. Then using the integrability assumption and stationary of $\Gamma$ we obtain for all $n\in\mathbb{N}$
\begin{eqnarray*}
 \bar\varrho_n(P,Q) &\le& \Gamma_n(\varrho) = E\int\varrho(gX, gY) \lambda_n(dg)\\
&=& \int E \varrho(gX,gY)\lambda_n(dg) = \Gamma(\varrho)\le\bar\varrho_{\mathrm s}(P,Q)+\varepsilon.
\end{eqnarray*}
This implies that $\bar\varrho(P,Q)\le\bar\varrho_{\mathrm s}(P,Q)$.

For the converse direction we choose for fixed $\varepsilon > 0$ and $n\ge 0$ an element $\Gamma_n\in M(P,Q)$ such that $\Gamma_n(\varrho)\le \bar\varrho_n(P,Q)+\varepsilon$. We define probability measures  $\{\bar\Gamma_n\}$ by
\begin{equation}\label{rue-8}
 \bar\Gamma_n(A):=\int_G \Gamma_n(gA)d\lambda_n (g).
\end{equation}
Then using the integrability condition and amenability of $G$ we obtain that
\begin{equation}\label{rue-9}
 \bar\Gamma_n(gA)-\bar\Gamma_n(A)=\int_G(\Gamma_n(gA)-\Gamma_n(A))\lambda_n (dg) \to 0,
\end{equation}
i.e. $\bar\Gamma_n$ is asymptotically left invariant on $ B \times B $.

By definition $\bar\Gamma_n\in M(P,Q)$, just take projections on finite components of $\bar\Gamma_n$
\begin{eqnarray*}
 \bar\Gamma_n(A_1\times\Omega)&=& \int_G \Gamma_n(gA_1\times\Omega)\lambda_n(dg)\\
&=& \int_G P(gA_1)\lambda_n (dg) = P(A_1)
\end{eqnarray*}
since $P$ is stationary. Using tightness of $\{\bar\Gamma_n\}$ we get a weakly converging subsequence of $\{\bar\Gamma_n\}$. W.l.g. we assume that $\{\bar\Gamma_n\}$ converges weakly to some probability measure  $\bar\Gamma$ on $ B \times  B $. In consequence by  \eqref{rue-9} we get $\bar\Gamma\in M_{\mathrm s}(P,Q)$. Finally,
\begin{eqnarray*}
 \bar\varrho_{\mathrm s} (P,Q) &\le& \bar\Gamma(\varrho) \le \lim\sup \bar\Gamma_n(\varrho) \\
&\le& \limsup \bar\varrho_n(P,Q)+\varepsilon \le \bar\varrho(P,Q)+\varepsilon
\end{eqnarray*}
for all $\varepsilon>0$ which concludes the proof.
\end{proof}

Motivated by the representation results in Theorem \ref{theo:3.3}, \ref{theo:rue-1} we now consider the optimal stationary coupling problem for general groups $G$ acting on $E=\real$.
Let $F$ be a finite subset of $ G $ and let $f:\real^F\to\real$ be a convex function.
The function $f$ is naturally identified with a function on $\Omega$ by $f(x)=f((x_g)_{g\in F})$.
As in Section \ref{sec:2} any choice of the subgradient of $f$ is denoted by $((\partial_g f)(x))_{g\in F}$.
Define an equivariant Borel measurable function $S:\Omega\to\Omega$ by the shifted sum of gradients
\begin{equation}
\label{eq:3.1}
 S_e(x) = \sum_{g\in F}(\partial_gf)(gx)
 \quad \mbox{ and } \quad
 S_h(x) = S_e(h^{-1}x), h\in G.
\end{equation}
Note that $S_e(x)$ depends only on $(x_g)_{g\in G(F)}$, where $G(F)$ is the subgroup generated by $F$ in $ G $.
We have $S\circ g=g\circ S$ for any $g\in G $ because
\[
S_h(gx)=S_e(h^{-1}gx)=S_{g^{-1}h}(x)=(gS(x))_h.
\]
Hence if $X$ is a stationary random field, then $(X,S(X))$ is a jointly stationary random field.

We obtain the following theorem.

\begin{theorem}\label{theo:3.1}
Let $P$, $Q$ be stationary random field probability measures with respect to a general group of measurable transformations $G$. 
 Let $S$ be an equivariant map as defined in \eqref{eq:3.1} with a convex function $f$.
 Let $X$ be a real stationary random field with law $\mu$ and assume that $X_e$ and $(\partial_gf(X))_{g\in F}$ are in $L^2(\mu)$.
 Then  $(X,S(X))$ is an optimal stationary coupling  w.r.t. squared distance between $\mu$ and $\mu^S$, i.e.
 \[
  E(X_e-S_e(X))^2 = \min_{(X,Y)\sim \Gamma\in M_{\rm s}(\mu, \mu^S)}\mathrm{E}[(X_e-Y_e)^2].
 \]
\end{theorem}

\begin{proof}
 The construction of the equivariant mapping in \eqref{eq:3.1} and the following remark allow us to transfer the proof of Theorem \ref{theo:2.3} to the class of random field models. 
 Fix $\Gamma\in M_{\rm s}(\mu,\mu^S)$.
 Let $G(F)$ be the subgroup generated by $F$ in $ G $.
 Then $G(F)$ is countable (or finite).
 We denote the restricted measure of $\mu$ on $\real^{G(F)}$ by $\mu|_{G(F)}$.
 By the gluing lemma,
 we can consider a jointly stationary random field $(X_g,Y_g,\tilde{X}_g)_{g\in G(F)}$
 on a common probability space
 such that $(X_g)_{g\in G(F)}\sim \mu|_{G(F)}$, $Y_g=S_g(X)$ and
 $(\tilde{X}_g,Y_g)_{g\in G(F)}\sim \Gamma|_{G(F)}$.
 Then we have
 \begin{eqnarray*}
  \frac{1}{2}\mathrm{E}[(X_e-S_e(X))^2-(\tilde{X}_e-S_e(X))^2] &=& \mathrm{E}[S_e(X)(\tilde{X}_e-X_e)]
  \\
  &=& \sum_{g\in F}\mathrm{E}\big[
  \big((\partial_gf)(gX)\big) (\tilde{X}_e-X_e)
  \big]
  \\
  &=& \sum_{g\in F}\mathrm{E}\big[
  \big((\partial_gf)(X)\big) (\tilde{X}_g-X_g)
  \big]
  \\
  &\le& \mathrm{E}[f(\tilde{X})-f(X)]
  \\
  &=& 0.
 \end{eqnarray*}
This implies that $(X,S(X))$ is an optimal stationary coupling w.r.t. squared distance between the random fields $\mu$ and $\mu^S=S_\#\mu$.
\end{proof}

\section{Optimal stationary couplings for general cost functions}
\label{sec:4}

The Monge--Kantorovich problem and the related characterization of optimal couplings have been generalized to general cost functions $c(x,y)$ in \cite{Rueschendorf-1991, Rueschendorf-1995}, while \cite{McCann-2001} extended the squared loss case to manifolds; see also the surveys in \cite{Rachev-Rueschendorf-1998} and \cite{Villani-2003,Villani-2009}. 
Based on these developments we will extend the optimal stationary coupling results in Sections \ref{sec:2}, \ref{sec:3} to more general classes of distance functions. Some of the relevant notions from transportation theory are collected in the Appendix \ref{sec:C}. We will restrict to the case of time parameter $\integer$. As in Section \ref{sec:3} an extension to random fields with general \emph{time} parameter is straightforward.

Let $E_1, E_2$ be Polish spaces. and let $c:E_1\times E_2\to\real$ be a measurable cost function. For $f:E_1\to\real$ and $x_0\in E_1$ let
\begin{equation}\label{eq:4.1}
 \partial^c f(x_0) = \big\{y_0\in E_2\mid c(x_0,y_0)-f(x_0)= \inf_{z_0\in E_1} \{c(z_0,y_0)-f(z_0)\} \big\}
\end{equation}
denote the set of \emph{$c$-supergradients} of $f$ in $x_0$.

A function $\varphi:E_1\to\real\cup\{-\infty\}$ is called \emph{$c$-concave} if there exists a function\linebreak $\Psi:E_2\to\real\cup\{-\infty\}$ such that
\begin{equation}\label{eq:4.2}
 \varphi(x)=\inf_{y\in E_2} (c(x,y)-\Psi(y)),\quad\forall x\in E_1.
\end{equation}
If $\varphi(x)=c(x,y_0)-\Psi(y_0)$, then $y_0\in\partial^c\varphi(x)$ is a $c$-supergradient of $\varphi$ at $x$. For squared distance $c(x,y)=\|x-y\|_2^2$ in $\real^m=E_1=E_2$ $c$-concavity of $\varphi$ is equivalent to the concavity of $\varphi-\|x\|_2^2/2$.

The characterization of optimal couplings $T(x)\in\partial^c\varphi(x)$ for some $c$-concave function $\varphi$ leads for regular $\varphi$ to a differential characterization of $c$-optimal coupling functions $T$
\begin{equation}\label{eq:4.3}
 \nabla_x c(x,T(x))=\nabla \varphi(x)
\end{equation}
see \cite{Rueschendorf-1991}, \cite{Villani-2009}. In case \eqref{eq:4.3} has a unique solution in $T(x)$ this equation describes optimal $c$-coupling functions $T$ in terms of differentials  of $c$-concave functions $\varphi$ and the set of $c$-supergradients $\partial^c \varphi(x)$ reduces to just one element
\begin{equation}\label{eq:4.4}
 \partial^c\varphi(x)=\{x-\nabla_x c^*(x,\varphi(x))\}.
\end{equation}

Here $c^*$ is the Legendre transform of $c$ and $\nabla_x c(x,\cdot)$ is invertible and $(\nabla_xc)^{\kern-.6pt -1}(x,\varphi(x))$ $=\nabla_x c^*(x,\varphi(x))$ (see \cite{Rueschendorf-1991, Rachev-Rueschendorf-1998} and \cite{Villani-2003, Villani-2009}). For functions $\varphi$ which are not $c$-concave, the supergradient $\partial^c\varphi(x)$ may be empty.

The construction of optimal stationary $c$-couplings of stationary processes can be pursued in the following way. Define the average distance per component $c_n:E_1^n\times E_2^n\to\real$ by
\begin{equation}\label{eq:4.5}
 c_n(x,y)=\frac1n \sum_{t=0}^{n-1} c(x_t,y_t)
\end{equation}
and assume that for some function $f:E_1^n\to\real$, there exists a function $F^n:E_1^n\to E_2^n$ such that
\begin{equation}\label{eq:4.6}
 F^n(x)=(F_k(x))_{0\le k\le n-1} \in \partial^{c_n} f(x),\quad x\in E_1^n.
\end{equation}
Note that \eqref{eq:4.6} needs to be satisfied only on the support of (the projection of) the stationary measure $\mu$. In general we can expect $\partial^{c_n} f(x)\not=\emptyset$, $\forall x\in E_1^n$ only if $f$ is $c_n$-concave. 
For fixed $y_0,\dots, y_{n-1} \in E_2$ we introduce the function $h_c(x_0)=\frac1n \sum_{k=0}^{n-1} c(x_0,y_k)$, $x_0\in E_1$. $h_c(x)$ describes the average distance of $x_0$ to the $n$ points $y_0, \dots, y_{n-1}$ in $E_2$. We define an equivariant map $S:E_1^\integer\to E_2^\integer$ by
\begin{equation}\label{eq:4.7}
 \begin{split}
  S_0(x) &\in \partial^c (h_c(x_0))\mid_{y_k=F_k(x_{-k},\dots, x_{-k+n-1} ),0\leq k\leq n-1}\\
S_t(x) &= S_0(L^{-t}x), \quad S(x)=(S_t(x))_{t\in\integer}.
 \end{split}
\end{equation}
Here the $c$-supergradient is taken for the function $h_c(x_0)$ and the formula is evaluated at $y_k=F_k(x_{-k},\dots,x_{-k+n-1})$, $0\le k\le n-1$. After these preparations we can state the following theorem.

\begin{theorem}[Optimal stationary \boldmath $c$-couplings of stationary processes]
\label{theo:4.1}
 Let $X=$ \linebreak $(X_t)_{t\in\integer}$ be a stationary process with values in $E_1$ and with distribution $\mu$, let $c:E_1\times E_2 \to\real$ be a measurable distance function on $E_1\times E_2$ and let $f:E_1^n\to\real$ be measurable $c_n$-concave. If $S$ is the equivariant map induced by $f$ in \eqref{eq:4.7} and if $c(X_0,S_0(X))$, $\{c(X_k,F_k(X^n))\}_{k=0}^{n-1}$ and $f(X^n)$ are integrable, then $(X,S(X))$ is an optimal stationary $c$-coupling of the stationary measures $\mu$, $\mu^S$ i.e.
\begin{equation}\label{eq:4.8}
 Ec(X_0,S_0(X))=\inf\{Ec(Y_0,Z_0)\mid (Y,Z)\sim\Gamma\in M_{\rm s}(\mu,\mu^S)\} = \bar c_{\rm s} (\mu,\mu^S).
\end{equation}
\end{theorem}

\begin{proof}
The construction of the equivariant function in \eqref{eq:4.7} allows us to extend the basic idea of the proof of Theorem \ref{theo:2.2} to the case of general cost function.   Fix any $\Gamma\in M_{\rm s}(\mu,\mu^S)$.
 By the gluing lemma,
 we can consider a jointly stationary process $(X,Y,\tilde{X})$ on a common probability space
 with properties $X\sim \mu$, $Y=S(X)$ and $(\tilde{X},Y)\sim \Gamma$.
 Then we have by construction in \eqref{eq:4.7} and using stationarity of $X$
 \begin{eqnarray*}
  \lefteqn{\mathrm{E}[c(X_0,S_0(X))-c(\tilde{X}_0,S_0(X))]} \quad
  \\
  &\le& \mathrm{E}\left[
  n^{-1}\sum_{k=0}^{n-1}\left.\{c(X_0,y_k)-c(\tilde{X}_0,y_k)\}\right|_{y_k=F_k(X_{-k},\ldots,X_{-k+n-1})}
  \right]
  \\
  &=& \mathrm{E}\left[
  n^{-1}\sum_{k=0}^{n-1}\left.\{c(X_k,y_k)-c(\tilde{X}_k,y_k)\}\right|_{y_k=F_k(X_0,\ldots,X_{n-1})}
  \right]
  \\
  &=& \mathrm{E}\left[
  c_n(X^n,F^n(X^n)) - c_n(\tilde{X}^n,F^n(X^n))
  \right]
  \\
  &\le& \mathrm{E}[f(X^n)-f(\tilde{X}^n)]
  \\
  &=& 0.
 \end{eqnarray*}
The last inequality follows from $c_n$-concavity of $f$ while the last equality is a consequence of the assumption that $X\stackrel{d}{=}\tilde X$. As consequence we obtain that $(X,S(X))$ is an optimal stationary $c$-coupling.
\end{proof}


The conditions in the construction \refeq{eq:4.7} of optimal stationary couplings in Theorem \ref{theo:4.1} (conditions \refeq{eq:4.6}, \refeq{eq:4.7}) simplify essentially in the case $n=1$. In this case we get as corollary of Theorem \ref{theo:4.1}
\begin{cor}\label{cor:4.2}
 Let $X=(X_t)_{t\in U}$ be a stationary process with values in $E_1$ and distribution $\mu$ and let $c:E_1\times E_2\to\real$ be a cost function as in Theorem \ref{theo:4.1}. Let $f:E_1\to\real$ be measurable $c$-concave and define
\begin{equation}\label{eq:4.8-add1}
 S_0(x)\in\partial^c f(x_0),  \quad S_t(x)=S_0(L^{-t}x)\in \partial^c f(x_t), \quad  S(x)=(S_t(x))_{t\in \integer}.
\end{equation}
Then $(X,S(X))$ is an optimal stationary $c$-coupling of the stationary measures $\mu$, $\mu^S$.
\end{cor}

Thus the equivariant componentwise transformation of a stationary process by supergradients of a $c$-concave function is an optimal stationary coupling. In particular in the case that $E_1=\real^k$ several examples of $c$-optimal transformations are given in \cite{Rueschendorf-1995} resp. \cite{Rachev-Rueschendorf-1998} which can be used to apply Corollary \ref{cor:4.2}.

In case $n\ge 1$ conditions \refeq{eq:4.6}, \refeq{eq:4.7} are in general not obvious. In some cases $c_n$-convexity of a function $f:E_1^n\to\real$ is however easy to see.

\begin{lemma}\label{lem:4.2}
 Let $f(x)=\sum_{k=0}^{n-1} f_k(x_k)$, $f_k:E_1\to\real$, $0\le k\le n-1$. If $f_k$ are $c$-concave, $0\le k\le n-1$, then $f$ is $c_n$-concave and
\begin{equation}\label{eq:4.8-add2}
 \partial^{c_n} f(x) = \sum_{k=0}^{n-1} \partial^c f(x_k).
\end{equation}
\end{lemma}

\begin{proof}
 Let $y_k\in\partial^c f_k(x_k)$, $0\le k\le n-1$, then with $y=(y_k)_{0\le k\le n-1}$ by definition of $c$-supergradients
\[
 c_n(x,y)-f(x)=\frac{1}{n}\sum_k(c(x_k,y_k)-f_k(x_k)) = \inf\{c_n(z,y)-f(z);z\in E_1^n\}
\]
and thus $y\in\partial^{c_n}f(x)$. The converse inclusion is obvious.
\end{proof}

Lemma \ref{lem:4.2} allows to construct some examples of functions $F^n$ satisfying condition \refeq{eq:4.5}. For $n>1$ non-emptiness of the $c$-supergradient of $h_c(x_0)=\frac1n \sum_{k=0}^{n-1} c(x_0,y_k)$ has to be established. The condition $u_0\in\partial^c h_c (x_0)$ is equivalent to
\begin{equation}\label{eq:4.8-2add1}
 c(x_0,u_0)-h_c(x_0) = \inf_z(c(z,u_0)-h_c(z)).
\end{equation}
In the differentiable case \refeq{eq:4.8-2add1} implies the necessary condition
\begin{equation}\label{eq:4.8-2add2}
 \nabla_x c(x_0,u_0) = \nabla_x h_c(x_0) = \frac1n \sum_{k=0}^{n-1} \nabla_x c(x_0,y_k).
\end{equation}
If the map $u\to\nabla_x c(x_o,u)$ is invertible then equation \refeq{eq:4.8-2add2} implies
\begin{equation}\label{eq:4.8-2add3}
 u_0 = (\nabla_x c)^{-1} (x_0,\cdot) \left(\frac1n \sum_{k=0}^{n-1} \nabla_x c(x_0,y_k)\right)
\end{equation}
(see \refeq{eq:4.4}). Thus in case that \refeq{eq:4.8-2add1} has a solution, it is given by \refeq{eq:4.8-2add3}.

\begin{lemma}\label{lem:4.3}
 If \refeq{eq:4.8-2add1} has a solution and $u\to \nabla_x c(x_0,u)$ is invertible, then for $x_0\in E_1$ $u_0=(\nabla_x c)^{-1} (x_0,\cdot) \left(\frac1n \sum_{k=0}^{n-1} \nabla_x c(x_0,y_k)\right)$ is a supergradient of $h_c$ in $x_0$,
\begin{equation}\label{eq:4.8-2add4}
 u_0\in\partial^c h_c(x_0).
\end{equation}
\end{lemma}

\begin{Example}\label{ex:1-add}
If $c(x,y)=H(x-y)$ for a strict convex function $H$, then $\nabla_x(c(x,\cdot)$ is invertible and we can construct the necessary $c$-supergradients of $h_c$. If for example $c(x,y)=\|x-y\|^2$, then we get for any $x_0\in\real^k$,
\begin{equation}\label{eq:4.8-2add5}
 u_0=u_0(x_0)=\frac1n \sum_{k=0}^{n-1} y_k = \overline y
\end{equation}
is independent of $x_0$ and
\begin{equation}\label{eq:4.8-2add6}
 \overline y\in\partial^c h_c(x_0), \quad \forall x_0\in\real ^k.
\end{equation}

If $c(x,y)=\|x-y\|^p$, $p>1$, then we get for $x_0\in\real^k$
\begin{equation}\label{eq:4.8-2add7}
 u_0=u_0(x_0)=x_0+|h(x_0)|^{^{\frac{1}{p-1}}} \frac{h(x_0)}{|h(x_0)|},
\end{equation}
where $h(x_0)=\frac1n \sum_{k=0}^{n-1} \|x_0-y_k\|^{p-1} \frac{x_0-y_k}{\|x_0-y_k\|}$. For this and related further examples see \cite{Rueschendorf-1995}.
\end{Example}

The $c$-concavity of $h_c$ has a geometrical interpretation. $u_0\in\partial^c h_c(x_0)$ if the difference of the distance of $z_0$ in $E_1$ to $u_0$ in $E_2$ and the average distance of $z_0$ to the given points $y_0,\dots,y_{n-1}$ in $E_2$ is minimized in $x_0$. The $c$-concavity of $h_c$ can be interpreted as a positive curvature condition for the distance $c$. To handle this condition we introduce the notion of convex stability.

\begin{definition}\label{def:4.2}
The cost function $c$ is called \emph{convex stable} of index $n\ge 1$ if for any $y\in E_2^n$
\begin{equation}\label{eq:4.9}
 h_c(x_0)=\frac1n\sum_{k=0}^{n-1} c(x_0,y_k), \quad x_0\in E_1, \quad\text{is $c$-concave}.
\end{equation}
$c$ is called convex stable if it is convex stable of index $n$ for all $n\ge 1$.
\end{definition}

\begin{Example}\label{ex:4.1}
 Let $E_1=E_2=H$ be a Hilbert space, as for example $H=\real^m$, let $c(x,y)=\|x-y\|^2/2$ and fix $y\in H^n$, then
\begin{eqnarray}\label{eq:4.10}
 h_c(x_0) &=& \frac1n \sum_{k=0}^{n-1} c(x_0,y_k) \nonumber\\  
&=&c(x_0,\bar y)+\frac1n \sum_{k=0}^{n-1} c(\bar y, y_k),
\end{eqnarray}
where $\bar y=\frac1n\sum_{k=0}^{n-1} y_k$ Thus by definition \eqref{eq:4.2} $h_c$ is $c$-concave and a $c$-supergradient of $h_c$ is given by $\bar y$ independent of $x_0$, i.e. 
\begin{equation}\label{eq:4.11}
 \bar y\in\partial^c h_c (x_0), \quad \forall x_0\in H.
\end{equation}
Thus the squared distance $c$ is convex stable.
\end{Example}

The property of a cost function to be convex stable is closely connected with the geometric property of non-negative cross curvature. 
Let $E_1$ and $E_2$ be open connected subsets in $\real^m$ ($m\geq 1$) with coordinates $x=(x^i)_{i=1}^m$ and $y=(y^j)_{j=1}^m$.
Let $c:E_1\times E_2\to\real$ be $C^{2,2}$, i.e. $c$ is two times differentiable in each variable.
Denote the cross derivatives by $c_{ij,k}=\partial^3c/\partial x^i\partial x^j\partial y^k$ and so on.
Define $c_x(x,y)=(\partial c/\partial x^i)_{i=1}^m$, $c_y(x,y)=(\partial c/\partial y^j)_{j=1}^m$, 
$U=\{c_x(x,y)\mid y\in E_2\}\subset\real^m$, $V=\{c_y(x,y)\mid x\in E_1\}\subset\real^m$.
Assume the following two conditions.
\begin{itemize}
\item[{[B1]}] The map $c_x(x,\cdot):E_2\to U$ and $c_y(\cdot,y):E_1\to V$ are diffeomorphic, i.e.,
they are injective and the matrix $(c_{i,j}(x,y))$ is positive definite everywhere.
\item[{[B2]}] The sets $U$ and $V$ are convex.
\end{itemize}
The conditions [B1] and [B2] are called bi-twist and bi-convex conditions, respectively.
Now we define the \emph{cross curvature} $\sigma(x,y;u,v)$ in $x\in E_1$, $y\in E_2$, $u\in\real^m$ and $v\in\real^m$ by
\begin{equation}\label{eq:4.12}
\sigma(x,y;u,v):= \sum_{i,j,k,l}\left(-c_{ij,kl} + \sum_{p,q}c_{ij,q}c^{p,q}c_{p,kl}\right)u^iu^jv^kv^l
\end{equation}
where $(c^{i,j})$ denotes the inverse matrix of $(c_{i,j})$.

The following result is given by \cite{Kim-McCann-2008}. Note that these authours use the terminology
{\em time-convex sliding-mountain} instead of the notion convex-stability as used in this paper.

\begin{propos}\label{prop:4.3}
  Assume the conditions [B1] and [B2].
  Then $c$ is convex stable if and only if the cross curvature is nonnegative, i.e.,
%
\begin{equation}\label{eq:4.13}
 \sigma(x,y;u,v)\ge 0,\quad \forall x,y,u,v.
\end{equation}
\end{propos}

The cross-curvature is related to the Ma-Trudinger-Wang tensor (\cite{Ma-Trudinger-Wang-2005}), which is the restriction of $\sigma(x,y;u,v)$ to $u^iv^jc_{i,j}=0$.
Known examples that have non-negative cross-curvature
are the $n$-sphere
(\cite{Kim-McCann-2008}, \cite{Figalli-Rifford-2009}), its perturbation
(\cite{Delanoe-Ge-2010}, \cite{Figalli-Rifford-Villani-2010}), their tensorial product and their Riemannian submersion.

If $E_1, E_2\subset\real$,
then the conditions [B1] and [B2] are implied from a single condition in case $c_{x,y}=\partial^2c(x,y)/\partial x\partial y\neq 0$.
Hence we have the following result as a corollary. A selfcontained simplified proof of this result is given in Appendix \ref{sec:D}.

\begin{propos}\label{prop:4.4}
  Let $E_1, E_2$ be open intervals in $\real$ and let $c\in C^{2,2}$, $c:E_1\times E_2\to\real$. Assume that $c_{x,y}\not= 0$ for all $x,y$. Then $c$ is convex stable if and only if $\sigma(x,y):=-c_{xx,yy}+c_{xx,y}c_{x,yy}/c_{x,y}\geq 0$.
\end{propos}

\begin{Example}\label{ex:4.2}
 Let $E_1,E_2\subset \real$ be open intervals and let $E_1\cap E_2=\emptyset$. Consider $c(x,y)=\frac1p|x-y|^p$ with $p\ge 2$ or $p<1$. Then $c$ is convex stable. In fact $c_{x,y}=-(p-1)|x-y|^{p-2}\not= 0$ for all $x,y$ and $\sigma(x,y)=(p-1)(p-2)|x-y|^{p-4}\ge 0$ for all $x,y$. As $p\to 0$, we also have a convex stable cost $c(x,y)=\log |x-y|$.

If the cost function $c$ is a metric then the optimal coupling in the case $E_1=E_2=\real$ can be reduced to the case of $E_1\cap E_2=\emptyset$ as in the classical Kantorovich--Rubinstein theorem. This is done by subtracting (and renormalizing) from the marginals $\mu_0$, $\nu_0$ the lattice infimum, i.e. defining 
\begin{equation}
\mu'_0:=\frac1a (\mu_0-\mu_0\wedge \nu_0), \quad 
\nu'_0:=\frac1a (\nu_0-\mu_0\wedge \nu_0).
\end{equation}
The new probability measures live on disjoint subsets to which the previous proposition can be applied.
\end{Example}

Some classes of optimal $c$-couplings for various distance functions $c$ have been discussed in \cite{Rueschendorf-1995}, see also \cite{Rachev-Rueschendorf-1998}. The examples discussed in these papers can be used to establish $c_n$-concavity of $f$ in some cases. This is an assumption used in Theorem \ref{theo:4.1} for the construction of the optimal stationary couplings.
Note that $c_n$ is convex-stable if $c$ is convex-stable. Therefore the following proposition due to \cite{Figalli-Kim-McCann-2010} (partially \cite{sei-2010c}) is also useful to construct a $c_n$-concave function $f$.

\begin{propos}\label{propos:4.11}
 Assume [B1] and [B2]. Then $c$ satisfies the non-negative cross curvature condition
 if and only if the space of $c$-concave functions is convex, that is, $(1-\lambda)f+\lambda g$ is $c$-concave as long as $f$ and $g$ are $c$-concave and $\lambda\in[0,1]$.
\end{propos}

\begin{Example}
 Consider Example~\ref{ex:4.2} again.
 Let $E_1=(0,1)$, $E_2=(-\infty,0)$, $c(x_1,y_1)=p^{-1}(x_1-y_1)^p$ ($p\geq 2$) and $c_n(x,y)=(np)^{-1}\sum_{k=0}^{n-1}(x_k-y_k)^p$.
 An example of $c_n$-concave functions of the form $f(x)=\sum_{k=0}^{n-1}f_k(x_k)$ with suitable real functions $f_k$ is given in \cite{Rueschendorf-1995} Example 1 (b).
 We add a further example here.
 Put $\bar{x}=n^{-1}\sum_{k=0}^{n-1}x_k$ and let $f(x)=A(\bar{x})$ with a real function $A$.
 We prove $f(x)$ is $c_n$-concave if $A'\geq 1$ and $A''\leq 0$.
 For example, $A(\xi)=\xi+\sqrt{\xi}$ satisfies this condition.
 Equation (\ref{eq:4.3}) becomes
 \begin{equation}\label{eq:4.24}
  n^{-1}(x_i-y_i)^{p-1} = n^{-1}A'(\bar{x})
 \end{equation}
 which uniquely determines $y_i\in E_2$ since $A'\geq 1$ and $x_i\in E_1$.
 To prove $c_n$-concavity of $f$, it is sufficient to show convexity of $x\mapsto c_n(x,y)-f(x)$ for each $y$. Indeed, the Hessian is
 \[
 \delta_{ij}n^{-1}(p-1)(x_i-y_i)^{p-2} - n^{-2}A''(\bar{x})
 \succeq -n^{-2}A''(\bar{x})
 \succeq 0
 \]
 in matrix sense.
 Note that the set of functions $A$ satisfying $A'\geq 1$ and $A''\leq 0$
 is convex, which is consistent with Proposition~\ref{propos:4.11}.
 Therefore, any convex combination of $A(\bar{x})$ and the $c_n$-concave function $\sum_kf_k(x_k)$
 discussed above is also $c_n$-concave by Proposition~\ref{propos:4.11}.
\end{Example}


\appendix 
\section*{Appendix}
\section{Gluing lemma for stationary measures} \label{sec:A}

The gluing lemma is a well known construction of joint distributions. We repeat this construction in order to derive an extension to the gluing of jointly stationary processes. For given probability measures $P$ and $Q$ on some measurable spaces $ E_1$ and $E_2$, we denote the set of joint probability measures on $ E_1\times E_2$
with marginals $P$ and $Q$ by $M(P,Q)$.

\begin{lemma}[Gluing lemma]\label{lem:a1}
 Let $P_1$, $P_2$, $P_3$ be Borel probability measures on Polish spaces $ E_1, E_2, E_3$, respectively. 
 Let $P_{12}\in M(P_1,P_2)$ and $P_{23}\in M(P_2,P_3)$.
 Then there exists a probability measure $P_{123}$ on $E_1\times E_2\times E_3$ with marginals $P_{12}$ on $E_1\times E_2$ and $P_{23}$ on $E_2\times E_3$.
\end{lemma}

\begin{proof}
 Let $P_{1|2}(\cdot|\cdot)$ be the regular conditional probability measure such that
 \[
  P_{12}(A_1\times A_2) = \int_{A_2} P_{1|2}(A_1|x)P_2(\mathrm{d}x)
 \]
 and $P_{3|2}(\cdot|\cdot)$ be the regular conditional probability measure such that
\[
  P_{32}(A_3\times A_2) = \int_{A_2} P_{3|2}(A_3|x)P_2(\mathrm{d}x).
 \]
Then a measure $P_{123}$ uniquely defined by
 \[
  P_{123}(A_1\times A_2\times A_3) := \int_{A_2} P_{12}(A_1|x)P_{32}(A_3|x)P_2(\mathrm{d}x)
 \]
 satisfies the required condition.
\end{proof}

Next we consider an extension of the gluing lemma to stationary processes. We note that even if a measure $P_{123}$ on $E_1^\integer \times E_2^\integer \times E_3^\integer$ has stationary marginals $P_{12}$ on $E_1^\integer \times E_2^\integer$ and $P_{23}$ on $E_2^\integer \times E_3^\integer $, it is not necessarily true that $P$ is stationary. 
For example, consider the $\{-1,1\}$-valued fair coin processes $X=(X_t)_{t\in\integer}$ and $Y=(Y_t)_{t\in\integer}$ independently, and let $Z_t=(-1)^tX_tY_t$.
Then $(X,Y)$ and $(Y,Z)$ have stationary marginal distributions respectively, but $(X,Y,Z)$ is not jointly stationary because $X_tY_tZ_t=(-1)^t$.

For given stationary measures $P$ and $Q$ on some product spaces, let $M_{\rm s}(P,Q)$ be the jointly stationary measures with marginal distributions $P$ and $Q$ on the corresponding product spaces.

\begin{lemma}\label{lem:a2}
 Let $ E_1, E_2, E_3$ be Polish spaces.
 Let $P_1,P_2,P_3$ be stationary measures on $ E_1^\integer , E_2^\integer , E_3^\integer $, respectively.
 Let $P_{12}\in M_{\rm s}(P_1,P_2)$ and $P_{23}\in M_{\rm s}(P_2,P_3)$.
 Then there exists a jointly stationary measure $P_{123}$ on $ E_1^\integer \times E_2^\integer \times E_3^\integer$ with marginals $P_{12}$ and $P_{23}$.
\end{lemma}

\begin{proof}
 One can apply the same construction as in the preceding lemma.
\end{proof}

\section{\boldmath $c$-concave function} \label{sec:C}

We review some basic results on $c$-concavity. See \cite{Rueschendorf-1991, Rueschendorf-1995, Rachev-Rueschendorf-1998, Villani-2003, Villani-2009} for details.

Let $E_1$ and $E_2$ be two Polish spaces and $c: E_1\times E_2\to\real $ be a measurable function.

\begin{definition}\label{def:d1}
 We define the  $c$-transforms of functions $f$ on $E_1$ and $g$ on $E_2$ by 
\[
f^c(y):= \inf_{x\in E}\{c(x,y)-f(x)\} %
\quad\text{and}\quad %
g^c(x):=\inf_{y\in E_2 }\{c(x,y)-g(y)\}.
\]
 A function $f$ on $ E_2$ is called $c$-concave if there exists some function $g$ on $E_2$ such that $f(x)=g^c(x)$.
\end{definition}

In general, $f^{cc}\geq f$ holds. Indeed, for any $x$ and $y$, we have $c(x,y)-f^c(y)\geq f(x)$. Then $f^{cc}(x)=\inf_{y}\{c(x,y)-f^c(y)\}\geq f(x)$.

\begin{lemma}\label{lem:c2}
 Let $f$ be a function of $ E_1$. Then $f$ is $c$-concave if and only if $f^{cc}=f$.
\end{lemma}

\begin{proof}
 The ``if'' part is obvious.
 We prove the ``only if'' part.
 Assume $f=g^c$.
 Then $f^c=g^{cc}\geq g$, and therefore
 \[
  f^{cc}(x)=\inf_y\{c(x,y)-f^c(y)\}\leq \inf_y\{c(x,y)-g(y)\} = g^c(x) = f(x).
 \]
 Since $f^{cc}\geq f$ always holds,
 we have $f^{cc}=f$.
\end{proof}

Define the $c$-supergradient of any function $f: E_1\to\real $ by
\[
 \partial^cf(x) = \left\{y\in E_2\mid c(x,y)-f(x)=f^c(y)\right\}.
\]

\begin{lemma} \label{lem:c3}
Assume that $\partial^cf(x)\neq \emptyset$ for any $x\in E_1$. Then $f$ is $c$-concave.
\end{lemma}

\begin{proof}
 Fix $x\in E_1$ and let $y\in\partial^cf(x)$.
 Then we have
 \[
  f(x)=c(x,y)-f^c(y)\geq f^{cc}(x)\geq f(x).
 \]
 Hence $f^{cc}=f$ and thus $f$ is $c$-concave.
\end{proof}

The converse of Lemma~\ref{lem:c3} does not hold in general. For example, consider $ E_1=[0,\infty)$, $E_2=\real$ and $c(x,y)=-xy$. Then $c$-concavity is equivalent to usual concavity. The function $f(x)=\sqrt{x}$ is concave but the supergradient at $x=0$ is empty.

\section{Proof of Proposition \ref{prop:4.4}} \label{sec:D}

Consider the cost function $c(x,y)$ on $E_1\times E_2$ with the assumptions in Proposition \ref{prop:4.4}.
Since $c_{x,y}\neq 0$, the map $y\mapsto c_x(x,y)$ is injective.
Denote its image and inverse function by $U=\{c_x(x,y)\mid y\in E_2 \}$ and $\eta_x=(c_x(x,\cdot))^{-1}:U\mapsto E_2 $, respectively.
Hence $c_x(x,\eta_x(u))=u$ for all $u\in U$ and $\eta_x(c_x(x,y))=y$ for all $y\in E_2 $.
Note that $U$ is an interval and therefore convex.
Also note that the subscript $x$ of $\eta_x$ does not mean the derivative.
By symmetry, we can define $V=\{c_y(x,y)\mid x\in E_1\}$ and $\xi_y=(c_y(\cdot,y))^{-1}:V\mapsto E_1$.

We first characterize the $c$-gradient of a differentiable $c$-concave function $f$.
Let $x\in E_1$ and $y\in\partial^cf(x)$.
Then $c(x,y)-f(x)\leq c(z,y)-f(z)$ for any $z\in E_1$.
By the tangent condition at $z=x$, we have $c_x(x,y)-f'(x)=0$, or equivalently, $y=\eta_x(f'(x))$.
Hence we have $\partial^cf(x)=\{\eta_x(f'(x))\}$.
We denote the unique element also by $\partial^cf(x)=\eta_x(f'(x))$.

To prove Proposition \ref{prop:4.4}, it is sufficient to show that the following conditions are equivalent:
\begin{itemize}
 \item[(i)] $c$ is convex stable for any index $n$
 \item[(ii)] The map $u\mapsto c(x,\eta_x(u))-c(z,\eta_x(u))$ is convex for all $x,z\in E_1$.
 \item[(iii)] $-c_{xx,yy}+c_{xx,y}c_{x,yy}/c_{x,y}\geq 0$.
\end{itemize}

We first prove (i) $\Leftrightarrow$ (ii).
Assume (i).
Let $\mathbb{Q}$ be the set of rational numbers.
By the definition of convex stability,
for any $u_0,u_1\in U$ and $\lambda\in[0,1]\cap\mathbb{Q}$,
the function
\[
 \phi(x):=(1-\lambda)c(x,\eta_x(u_0))+\lambda c(x,\eta_x(u_1))
\]
is $c$-concave. The $c$-gradient of $\phi$ is given by
\[
 \partial^c\phi(x)
 =\eta_x((1-\lambda)c_x(x,\eta_x(u_0))+\lambda c_x(x,\eta_x(u_1)))
 =\eta_x((1-\lambda)u_0+\lambda u_1).
\]
Then $c$-concavity, $c(x,\partial^c\phi(x)) - \phi(x) \leq c(z,\partial^c\phi(x)) - \phi(z)$ for any $z$, is equivalent to
\begin{eqnarray*}
\lefteqn{c(x,\eta_x((1-\lambda)u_0+\lambda u_1)) - c(z,\eta_x((1-\lambda)u_0+\lambda u_1))} \quad  \\
 &\leq& (1-\lambda)\{c(x,\eta_x(u_0))-c(z,\eta_x(u_0))\}
 +\lambda \{c(x,\eta_x(u_1))-c(z,\eta_x(u_1))\}.
\end{eqnarray*}
Since both hand side is continuous with respect to $\lambda$, (ii) is obtained.
The converse is similarly.

Next we prove (ii) $\Leftrightarrow$ (iii). Assume (ii).
Fix $x,z\in E_1$ and $u_0\in U $. Let $y_0=\eta_x(u_0)$ and therefore $u_0=c_x(x,y_0)$.
Since $u\mapsto c(x,\eta_x(u))-c(z,\eta_x(u))$ is convex for any $z$, its second derivative at $u=u_0$ is non-negative:
\begin{align}
 \lefteqn{\left.\partial_u^2\{c(x,\eta_x(u))-c(z,\eta_x(u))\}\right|_{u=u_0}} \quad
 \nonumber\\
 &=\{c_{yy}(x,y_0)-c_{yy}(z,y_0)\}(\eta_x^{(1)}(u_0))^2
 +\{c_y(x,y_0)-c_y(z,y_0)\}\eta_x^{(2)}(u_0)
 \nonumber\\
 &\geq 0.
 \label{eq:D.1}
\end{align}
On the other hand, by differentiating the identity $c_x(x,\eta_x(u))=u$ twice
at $u=u_0$, we have
\[
 c_{x,yy}(x,y_0)(\eta_x^{(1)}(u_0))^2
 + c_{x,y}(x,y_0)\eta_x^{(2)}(u_0) = 0.
\]
Combining the two relations, we have
\[
 \left[\{c_{yy}(x,y_0)-c_{yy}(z,y_0)\} - \frac{c_{x,yy}(x,y_0)}{c_{x,y}(x,y_0)}
 \{c_y(x,y_0)-c_y(z,y_0)\}
 \right](\eta_x^{(1)}(u_0))^2 \geq 0.
\]
Since $\eta_x^{(1)}(u_0)=1/c_{x,y}(x,y_0)\neq 0$, we obtain
\[
 \{c_{yy}(x,y_0)-c_{yy}(z,y_0)\} - \frac{c_{x,yy}(x,y_0)}{c_{x,y}(x,y_0)}
 \{c_y(x,y_0)-c_y(z,y_0)\} \geq 0.
\]
Now let $v_0=c_y(x,y_0)$ and $v=c_y(z,y_0)$.
Then $x=\xi_{y_0}(v_0)$ and $z=\xi_{y_0}(v)$ from the definition of $\xi_y$. We have
\begin{align}
 \{c_{yy}(\xi_{y_0}(v_0),y_0)-c_{yy}(\xi_{y_0}(v),y_0)\}
 - \frac{c_{x,yy}(\xi_{y_0}(v_0),y_0)}{c_{x,y}(\xi_{y_0}(v_0),y_0)}
 (v_0-v) \geq 0.
 \label{eq:D.2}
\end{align}
This means convexity of the map $v\mapsto -c_{yy}(\xi_{y_0}(v),y_0)$.
Hence its second derivative is non-negative.
Therefore
\[
 -c_{xx,yy}(z,y_0)(\xi_{y_0}^{(1)}(v))^2 - c_{x,yy}(z,y_0)\xi_{y_0}^{(2)}(v) \geq 0.
\]
On the other hand, by differentiating the identity $c_y(\xi_{y_0}(v),y_0)=v$ twice,
we have
\[
 c_{xx,y}(z,y_0)(\xi_{y_0}^{(1)}(v))^2 + c_{x,y}(z,y_0)\xi_{y_0}^{(2)}(v) = 0.
\]
Combining the two relations, we have
\begin{align}
 \left[-c_{xx,yy}(z,y_0) + \frac{c_{xx,y}(z,y_0)}{c_{x,y}(z,y_0)}c_{x,yy}(z,y_0)
 \right] (\xi_{y_0}^{(1)}(v))^2 \geq 0.
 \label{eq:D.3}
\end{align}
Since $\xi_{y_0}^{(1)}(v)=1/c_{x,y}(z,y_0)\neq 0$, we conclude
\[
 -c_{xx,yy}(z,y_0) + c_{x,yy}(z,y_0)\frac{c_{xx,y}(z,y_0)}{c_{x,y}(z,y_0)} \geq 0.
\]
Since $z$ and $y_0(=\eta_x(u_0))$ are arbitrary,
we obtain (iii).

The proof of (iii) $\Rightarrow$ (ii) is just the converse.
First, \eqref{eq:D.3} follows from (iii).
Since \eqref{eq:D.3} is the second derivative of
the left hand side of \eqref{eq:D.2},
the convexity condition \eqref{eq:D.2} follows.
The condition \eqref{eq:D.2} is equivalent to \eqref{eq:D.1}, and \eqref{eq:D.1} is also equivalent to (ii).
This completes the proof.


\bigskip
\bibliographystyle{plainnat}
{
%

\label{sec:bib}
}

\bigskip\bigskip

\noindent
\parbox[t]{.43\textwidth}{%
Ludger Rüschendorf\\
Mathematische Stochastik\\
University of Freiburg\\
Eckerstr. 1\\
79104 Freiburg\\
Germany\\
ruschen@stochastik.uni-freiburg.de} %
\hfill
\parbox[t]{.53\textwidth}{
Tomonari Sei\\
Department of Mathematics\\
Keio University\\
3-14-1 Hiyoshi, Kohoku-ku\\
Yokohama, 223-8522\\
Japan\\
sei@math.keio.ac.jp} %

\end{document}